\documentclass[12pt]{amsart}
\usepackage[bookmarksnumbered, plainpages,colorlinks=true,allcolors=blue]{hyperref}
\usepackage{amsmath}
\usepackage[english,  activeacute]{babel}
\usepackage[utf8]{inputenc}
\usepackage{amssymb}
\usepackage{amsthm}
\usepackage{graphics,graphicx}
\usepackage{enumerate}
\usepackage{array}
\usepackage{bm}

\usepackage{a4wide}
\usepackage{color, url}
\usepackage{float}
\usepackage{comment}
\usepackage{xcolor}

\usepackage[a4paper,margin=1.9cm,top=2.5cm,bottom=2.5cm,centering,vcentering]{geometry}

\allowdisplaybreaks
\newcommand{\sst}{\overline{S}}

\usepackage{mmacells}
\usepackage{mma} 

\newcommand{\seqnum}[1]{\href{http://oeis.org/#1}{\underline{#1}}}

\theoremstyle{plain}
\newtheorem{theorem}{Theorem}[section]

\newtheorem{corollary}[theorem]{Corollary}
\newtheorem{lemma}[theorem]{Lemma}

\theoremstyle{definition}
\newtheorem{definition}[theorem]{Definition}

\newtheorem{remark}{Remark}[section]

\def\modd#1 #2{#1\ \mbox{\rm (mod}\ #2\mbox{\rm )}}

\usepackage[usestackEOL]{stackengine}[2013-10-15]
\stackMath

\title{Arithmetic properties of $t$-Schur overpartitions}

\subjclass[2020]{11P83, 05A15, 05A17.}
\keywords{integer partitions, overpartitions, Ramanujan-type congruences, Schur partitions theorem}

\begin{document}

\author[Mohammed L. Nadji]{Mohammed L. Nadji}
\address{Faculty of Mathematics, University of Science and Technology Houari Boumediene, Algiers,
Algeria}
\address{RECITS laboratory, BP 32, El Alia, Bab Ezzouar, Algiers 16111, Algeria}
\email{m.nadji@usthb.dz}

\author[Manjil P. Saikia]{Manjil P. Saikia}
\address{Mathematical and Physical Sciences division, School of Arts \& Sciences, Ahmedabad University, Navrangpura, Ahmedabad 380009, Gujarat, India}
\email{manjil.saikia@ahduni.edu.in}

\author[James A. Sellers]{James A. Sellers}
\address{Department of Mathematics and Statistics, University of Minnesota Duluth, Duluth, MN, USA}
\email{jsellers@d.umn.edu}

\newcommand{\nadji}[1]{\mbox{}{\sf\color{green}[Nadji: #1]}\marginpar{\color{green}\Large$*$}} 

\newcommand{\manjil}[1]{\mbox{}{\sf\color{magenta}[Saikia: #1]}\marginpar{\color{magenta}\Large$*$}}

\newcommand{\sellers}[1]{\mbox{}{\sf\color{blue}[Sellers: #1]}\marginpar{\color{blue}\Large$*$}}

\begin{abstract} In a recent work, Nadji and Ahmia introduced the $t$-Schur overpartitions as an overpartition analogue for $t$-Schur partitions, which generalizes the classical Schur's partitions into parts congruent to $1$ or $5$ modulo $6$. We continue the study of this new class of overpartitions and prove several arithmetic results for the cases $t=3,9$ and $t$ being a power of $2$ or a power of $3$.

\end{abstract}

\maketitle

\section{Introduction}
\label{intro}

A \emph{partition} of a positive integer $n$ is a non-increasing sequence of positive integers \[\lambda = (\lambda_1, \lambda_2, \dots, \lambda_m)\] such that $n=\lambda_1+\lambda_2+\cdots+\lambda_m$. Each summand $\lambda_i$ is called a \emph{part}. An \emph{overpartition} of $n$ is a partition of $n$ in which  the first occurrence of parts of each size  may be overlined.  Let $\overline{p}(n)$ denote the number of overpartitions of $n$. Corteel and Lovejoy \cite{CorteelLovejoy2003} gave the generating function for the number of overpartitions, which is
\[
\sum_{n\geq 0} \overline{p}(n)q^{n}= \prod_{n\geq 1} \frac{1+q^{n}}{1-q^{n}}=\frac{(-q;q)_\infty}{(q;q)_\infty}=\frac{f_2}{f_1^2}.
\]
Here and throughout we will use the standard $q$-series notation
$$
(a;q)_{n}:=\begin{cases} 
      \prod\limits_{i=0}^{n-1}(1-aq^{i}), &  \text{if $n\geq0$;} \\
      1, & \text{if $n=0$.}
    \end{cases}
$$
Moreover, $(a;q)_{\infty}=\lim_{n\rightarrow \infty}(a;q)_{n}$, for $|q|<1$ and $f_{k}=(q^{k},q^{k})_{\infty}$.

The Rogers-Ramanujan identities \cite{RogersRamanujan1919} are two remarkable identities that establish deep connections between partition theory, combinatorics, modular forms, and representation theory. These identities were discovered independently by Rogers \cite{Rogers1894} in 1894  and later by Ramanujan in 1913 and Schur \cite{Schur1917}  in 1917. Following the discovery of these identities, in 1926, Schur \cite{Schur1926} presented his famous Rogers-Ramanujan type identity.
\begin{theorem}[Schur, 1926] \label{thmSchur}
Let $A(n)$ denote the number of partitions of $n$ into parts congruent to $ \pm 1$ modulo $6$, $B(n)$ denote the number of partitions of $n$ into distinct parts congruent to $ \pm 1$ modulo $3$, and $C(n)$ denote the number of partitions of $n$ into parts differ by at least $3$ where no consecutive multiples of $3$ appear. Then for all $n$, 
\[
A(n)=B(n)=C(n).
\]
\end{theorem}
As expected, Schur's theorem also became a highly influential partition identity, and several different proofs were provided using a variety of techniques. These techniques include bijections by Bressoud \cite{Bressoud1980} and Bessenrodt \cite{Bessenrodt1991}, the method of weighted words by Alladi and Gordon \cite{AlladiGordon1993}, and recurrences by Andrews \cite{Andrews1967a, Andrews1968, Andrews1971}.

In 1968 and 1969, Andrews \cite{Andrews1968a, Andrews1969} proved two partition theorems of the Rogers-Ramanujan type which generalize Schur's partition identity. Andrews' two generalizations of Schur's theorem went on to become two of the most influential results in the theory of partitions, finding applications in combinatorics \cite{Alladi1997, CorteelLovejoy2006, Yee2008}, representation theory \cite{AndrewsOlsson1991} and quantum algebra \cite{Oh2015}. Notably, Lovejoy \cite{Lovejoy2005} presented the overpartition analogue of Schur's theorem in the following result.
\begin{theorem}[Lovejoy, 2005] Let $A(n,k)$ denote the number of overpartitions of $n$ into parts congruent to $1$ or $2$ modulo $3$ with $k$ non-overlined parts. Let $B(n,k)$ denote the number of overpartitions $\lambda_{1}+\lambda_{2}+ \cdots + \lambda_{m}$ of $n$, having $k$ non-overlined parts and satisfying the difference conditions
\[
\lambda_{i}-\lambda_{i+1} \geq
\begin{cases}
0+3 \chi ( \overline{\lambda_{i+1}} ) & \text{if} \quad \lambda_{i+1} \equiv 1,2 \ (\mathrm{mod}\ 3),\\
1+3 \chi ( \overline{\lambda_{i+1}} ) & \text{if} \quad \lambda_{i+1} \equiv 0 \ (\mathrm{mod}\ 3),
\end{cases}
\]
where $\chi ( \overline{\lambda_{i+1}} )=1$ if $\lambda_{i+1}$ is overlined and $0$ otherwise. Then for all $k,n\geq 0$, 
\[
A(k,n)=B(k,n).
\]
\end{theorem}

Recently, Nadji and Ahmia \cite{NadjiAhmia2024} studied the partitions into parts that are simultaneously regular and distinct from both combinatorial and arithmetic perspectives, where the authors defined the $t$-Schur partitions of a positive integer $n$.
\begin{definition}[Nadji \& Ahmia, 2025]
For any odd positive integer $t\geq 3$, a $t$-Schur partition of $n$ is a partition into parts congruent to $i$ modulo $2t$, where $i\in I(t)$, and $I(t)=\lbrace 1, 3, 5, 7, \dots, (2t-1) \rbrace\backslash\lbrace t \rbrace$. 
\end{definition}
For $t=3$, this definition yields the classical Schur's theorem \cite[\seqnum{A003105}]{OEIS} (or $3$-Schur partitions), where $I(3)=\lbrace 1, 5\rbrace$. The number of $t$-Schur partitions of $n$, denoted by $S_{t}(n)$, satisfies the generating function
\[
\sum_{n \geq 0} S_{t}(n) q^{n}=\prod_{n\geq 1}\dfrac{(1-q^{t(2n-1)})}{(1-q^{2n-1})}=\dfrac{f_{2}f_{t}}{f_{1}f_{2 t}}.
\]
Moreover, the number $S_{t}(n)$ also counts the number of $(2,t)$-biregular partitions of $n$, which are partitions of $n$ into parts not divisible by $2$ and $t$.

Building upon this, Nadji and Ahmia \cite{NadjiAhmia2024} introduced the number of $t$-Schur overpartitions of a positive integer $n$, denoted by $\overline{S}_{t}(n)$, as a generalization of the $t$-Schur partitions. The generating function for $\overline{S}_{t}(n)$ is given by
\begin{equation}
 \sum_{n\geq 0} \overline{S}_{t}(n)q^{n}=\prod_{n\geq 1}\dfrac{(1+q^{2n-1})(1-q^{t(2n-1)})}{(1+q^{t(2n-1)})(1-q^{2n-1})}=\dfrac{f_{2}^{3}f_{t}^{2}f_{4t}}{f_{1}^{2}f_{4}f_{2t}^{3}}=\frac{\varphi(-q^2)\varphi(-q^t)}{\varphi(-q)\varphi(-q^{2t})},  \label{maineq}
\end{equation}
where $\varphi(-q)$ is Ramanujan's theta function and is defined as 
\[
\varphi(-q):=1+2\sum_{n\geq 1}(-1)^nq^{n^2}.\] The $t$-Schur partitions also belong to the broader family of $(\ell_{1}, \ell_{2})$-biregular overpartitions (for $\ell_{1}=2$ and $\ell_{2}\geq 3$ odd), which has been studied by Nadji, Ahmia and Ramírez \cite{NAR2025} and Alanazi, Munagi, and Saikia \cite{AMS2025} recently.

The authors in \cite{NadjiAhmia2024} also proved the following theorem using bijections.
\begin{theorem}[Nadji \& Ahmia, 2025]
For any odd  $t\geq 3$, let $I_{t}(2n)$ denote the number of partitions of $2n$ into parts indivisible by $t$, where the odd parts appear only with a multiplicity of $2$. Then, for all $n$, we have
\[
\overline{S}_{t}(n)=I_{t}(2n).
\]
\end{theorem}
\noindent For instance, the coefficients $I_{3}(2n)$ for $n\geq 0$ match the sequence \seqnum{A103260} in OEIS \cite{OEIS}. Moreover, Nadji and Ahmia \cite{NadjiAhmia2024} also proved the following congruences for $\overline{S}_{3}(n)$.
\begin{theorem}[Nadji \& Ahmia, 2025] For all $n\geq 0$,
\begin{align*}
&\overline{S}_{3}(6n+5) \equiv  0 \ (\mathrm{mod} \ 4),\\
&\overline{S}_{3}(12n+7) \equiv 0 \ (\mathrm{mod} \ 8),\\
&\overline{S}_{3}(12n+11) \equiv 0 \ (\mathrm{mod} \ 16).
\end{align*}
\end{theorem}

In this paper, we establish new arithmetic properties of the number $\overline{S}_{t}(n)$ for $t=3,9$, and investigate the arithmetic behavior of the generating function \eqref{maineq} in the case $t=2^k$ for $k\geq 3$ and $t=3^k$ for $k\geq 2$. We note that Ajeyakumar, Sumanth Bharadwaj, and Chandankumar \cite{arxiv} also studied the cases $t=3,9$ recently.

The rest of the paper is organized as follows. We collect in Section \ref{sec:prelim} all the essentials details required for our proofs in the later sections. In Section \ref{sec1}, we present a general congruence for the number $\overline{S}_t(n)$ modulo $4$ for any odd $t \geq 3$. In Section \ref{sec4}, we establish some congruences for the case $t=2^k$ with $k\geq 3$ and $t=3^k$ with $k\geq 2$, which motivates us to look at the specific cases of $t=3, 9$ more closely. Sections \ref{sec2} and \ref{sec3} are devoted to the study of the arithmetic behavior of $\overline{S}_t(n)$ for $t=3,9$, where we derive several congruences modulo $3,4,6,8,9$ and $12$. 

\section{Preliminary Results}\label{sec:prelim}

\subsection{Elementary Results}

In this subsection, we list a few dissection formulas that are useful in proving our main results. \emph{Ramanujan's general theta function} $f(a,b)$ \cite[p. 34, 18.1]{Berndt1991} is defined as
\[
f(a,b):=\sum_{n=-\infty}^{\infty}a^{n(n+1)/2}b^{n(n-1)/2}, \quad  |ab| <1.
\]
The product representation of $f(a,b)$ arises from the famous Jacobi triple product identity \cite[p. 35, Entry 19]{Berndt1991}
\[
f(a,b)=(-a,ab)_{\infty}(-b,ab)_{\infty}(ab,ab)_{\infty}.
\]
Some special cases of theta functions are denoted by 
\begin{align*}
&\varphi(q):=f(q,q)=1+2\sum_{n=1}^{\infty}q^{n^{2}}=(-q,q^{2})_{\infty}^{2}(q^{2},q^{2})_{\infty}=\frac{f_{2}^{5}}{f_{1}^{2}f_{4}^{2}}, \\
&\psi(q):=f(q,q^{3})=\sum_{n=0}^{\infty}q^{n(n+1)/2}=\frac{(q^{2},q^{2})_{\infty}}{(q,q^{2})_{\infty}}=\frac{f_{2}^{2}}{f_{1}},\\
& \varphi(-q)=1+2\sum_{n=1}^{\infty} (-1)^{n} q^{n^{2}}=\frac{f_{1}^{2}}{f_{2}}.
\end{align*}

\begin{lemma} \label{binomial} 
For all primes $p$ and all $k, m\geq 1$, we have 
\begin{equation}
 f_{pm}^{p^{k-1}} \equiv f_{m}^{p^{k}}  \ (\mathrm{mod} \ p^{k}). \label{binomiallemma} 
\end{equation}
\end{lemma}
\noindent Lemma \ref{binomial} easily follows from the binomial theorem.

\begin{lemma}\cite[Thm. 2.1]{CuiGu2013} \label{psidissectionlemma} For any odd prime $p$, we have
\begin{equation}
\psi(q)= \sum_{k=0}^{(p-3)/2} q^{k(k+1)/2} f(q^{\frac{p^{2}+(2k+1)p}{2}}, q^{\frac{p^{2}-(2k+1)p}{2}})+ q^{\frac{p^{2}-1}{8}}\psi(q^{p^{2}}).
 \label{psidissection}
\end{equation}
Furthermore, $\frac{m^{2}+m}{2}\not\equiv\frac{p^{2}-1}{8} \ (\mathrm{mod} \ p)$ for $0 \leq m \leq (p-3)/2$.
\end{lemma}

\begin{lemma}\cite[Thm. 2.2]{CuiGu2013} \label{fdissectionlemma} For any prime $p\geq 5$, we have
\begin{equation}
f_{1}= \sum_{\substack{k = (1-p)/2 \\ k\neq \frac{\pm p-1}{6}} }^{(p-1)/2} (-1)^{k} q^{k(3k+1)/2} f(-q^{\frac{3p^{2}+(6k+1)p}{2}}, -q^{\frac{3p^{2}-(6k+1)p}{2}})+(-1)^{\frac{\pm p-1}{6}} q^{\frac{p^{2}-1}{24}}f_{p^{2}},
 \label{fdissection}
\end{equation}
where 
\[
   	\dfrac{\pm p-1}{6}=
    \begin{cases}
        \dfrac{p-1}{6} & \text{if} \  p \equiv \ 1 \ (\mathrm{mod} \ 6), \\
      	\dfrac{-p-1}{6} & \text{if} \  p \equiv \ -1 \ (\mathrm{mod} \ 6).
    \end{cases}
\]
\end{lemma}

\begin{lemma} \label{lemma1} The following $2$-dissections hold.
\begin{align}
\frac{1}{f_1^4} &= \frac{f_4^{14}}{f_2^{14}f8^4}+4q\frac{f_4^2f_8^4}{f_2^{10}},\label{diss-2-f14}\\
\frac{f_{3}^{2}}{f_{1}^{2}} &=\frac{f_{4}^{4}f_{6}f_{12}^{2}}{f_{2}^{5}f_{8}f_{24}}+2q\frac{f_{4}f_{6}^{2}f_{8}f_{24}}{f_{2}^{4}f_{12}}, \label{lemma1.1}\\
\dfrac{f_{3}^{3}}{f_{1}} &=\dfrac{f_{4}^{3}f_{6}^{2}}{f_{2}^{2}f_{12}}+q\dfrac{f_{12}^{3}}{f_{4}}, \label{lemma1.2}\\
\frac{f_{1}}{f_{3}^{3}} &=\frac{f_{2} f_{4}^{2} f_{12}^{2}}{f_{6}^{7}} -q \frac{f_{2}^{3} f_{12}^{6}}{f_{4}^{2} f_{6}^{9}}, \label{lemma1.3}\\
\frac{f_{3}}{f_{1}^{3}} &=\frac{f_{4}^{6} f_{6}^{3}}{f_{2}^{9} f_{12}^{2}} +3q \frac{f_{4}^{2} f_{6} f_{12}^{2}}{f_{2}^{7}}, \label{lemma1.4}\\
\frac{f_{1}^{3}}{f_{3}} &=\frac{f_{4}^{3}}{f_{12}} -3q \frac{f_{2}^{2} f_{12}^{3} }{ f_{4} f_{6}^{2}}, \label{lemma1.41}\\
\dfrac{1}{f_{1}^{2} f_{3}^{2}} &= \dfrac{f_{8}^{5} f_{24}^{5}}{f_{2}^{5} f_{6}^{5} f_{16}^{2} f_{48}^{2}} +2q \dfrac{f_{4}^{4} f_{12}^{4}}{f_{2}^{6} f_{6}^{6}} +4q^{4} \dfrac{f_{4}^{2} f_{12}^{2} f_{16}^{4} f_{48}^{2}}{f_{2}^{5} f_{6}^{5} f_{8} f_{24}}, \label{lemma1.5}\\
\frac{1}{f_{1} f_{3}} &= \frac{f_{8}^{2} f_{12}^{5}}{f_{2}^{2} f_{4} f_{6}^{4} f_{24}^{2}}+q \frac{f_{4}^{5} f_{24}^{2}}{f_{2}^{4} f_{6}^{2} f_{8}^{2} f_{12}},\label{lemma1.6}\\
f_{1}f_{3} &= \frac{f_{2} f_{8}^{2} f_{12}^{4}}{f_{4}^{2} f_{6} f_{24}^{2}}-q \frac{f_{4}^{4} f_{6} f_{24}^{2}}{f_{2} f_{8}^{2} f_{12}^{2}}. \label{lemma1.7}
\end{align}
\end{lemma}

\begin{proof}
Equation \eqref{diss-2-f14} is \cite[Lemma 2.2]{SellersJMAA}, Yao and Xia \cite{XiaYao2013} gave a proof of Equation \eqref{lemma1.1}. Hirschhorn et al. \cite{HirschhornGarvanBorwein1993} proved  Equation \eqref{lemma1.2}. Equations \eqref{lemma1.3} and \eqref{lemma1.41} follow by changing $q$ to $-q$ in \eqref{lemma1.2} and \eqref{lemma4.1}, respectively, with $(-q,-q)_{\infty}=f_{2}^{3}/f_{1}f_{4}$. Baruah and Ojah \cite{BaruahOjah2012, BaruahOjah2015} proved \eqref{lemma1.4}, \eqref{lemma1.5}, \eqref{lemma1.6} and \eqref{lemma1.7}.
\end{proof}

\begin{lemma} \label{lemma2}  We have the following $3$-dissections.
 \begin{align}
    &f_{1}f_{2}= \frac{f_{6} f_{9}^{4}}{ f_{3} f_{18}^{2}}-qf_{9}f_{18} -2q^{2} \frac{f_{3} f_{18}^{4}}{f_{6} f_{9}^{2}}, \label{lemma2.1}\\
    &\dfrac{1}{f_{1}f_{2}}=\dfrac{f_{9}^{9}}{f_{3}^{6}f_{6}^{2}f_{18}^{3}}+q\dfrac{f_{9}^{6}}{f_{3}^{5}f_{6}^{3}}+3q^{2}\dfrac{f_{9}^{3}f_{18}^{3}}{f_{3}^{4}f_{6}^{4}}-2q^{3}\dfrac{f_{18}^{6}}{f_{3}^{3}f_{6}^{5}}+4q^{4}\dfrac{f_{18}^{9}}{f_{3}^{2}f_{6}^{6}f_{9}^{3}}. \label{lemma2.2}
    \end{align}
\end{lemma}
\noindent Lemma \ref{lemma2} was proved by Chan \cite{Chan2010}.

The following $5$-dissection formula appears in the papers of Ramanujan  \cite[p. 212]{RAM}. 
\begin{lemma} \label{lemma3} We have
    \begin{equation}
        f_{1}=f_{25}(a-q-q^{2}/a), \label{lemma3.1}
    \end{equation}
    where
$a=\dfrac{(q^{10};q^{25})_\infty(q^{15};q^{25})_\infty}{(q^{5};q^{25})_\infty(q^{10};q^{25})_\infty}$. 
\end{lemma}

\begin{lemma} \label{lemma4} The following $3$-dissections hold.
\begin{align}
&\dfrac{f_{4}}{f_{1}}=\dfrac{f_{12}f_{18}^{4}}{f_{3}^{3}f_{36}^{2}}+ q\dfrac{f_{6}^{2}f_{9}^{3}f_{36}}{f_{3}^{4}f_{18}^{2}}+2q^{2}\dfrac{f_{6}f_{18}f_{36}}{f_{3}^{3}}, \label{lemma4.1}\\
&\frac{f_{1}}{f_{4}}= \frac{f_{6} f_{9} f_{18}}{f_{12}^{3}}-q\frac{f_{3} f_{18}^{4}}{f_{9}^{2} f_{12}^{3}}-q^{2} \frac{f_{6}^{2} f_{9} f_{36}^{3}}{f_{12}^{4} f_{18}^{2}}. \label{lemma4.2}
\end{align}
\end{lemma}

\begin{proof}
Hirschhorn \cite[(33.2.6)]{Hirschhorn2017}  proved Equation \eqref{lemma4.1}. Equation \eqref{lemma4.2} corresponds to Lemma  1.2.23 in \cite{Nayaka2018}.
\end{proof}

\begin{lemma} \label{lemma5} We have the following $2$-dissections.
\begin{align}
    &\frac{f_{9}}{f_{1}}=\frac{f_{12}^{3}f_{18}}{f_{2}^{2} f_{6} f_{36}} +q \frac{f_{4}^{2} f_{6} f_{36}}{f_{2}^{3} f_{12}}, \label{lemma5.1}\\
    &\frac{f_{9}^{2}}{f_{1}^{2}}=\frac{f_{12}^{6}f_{18}^{2}}{f_{2}^{4} f_{6}^{2} f_{36}^{2}} +2q \frac{f_{4}^{2} f_{12}^{2} f_{18}}{f_{2}^{5}}
+q^{2} \frac{f_{4}^{4} f_{6}^{2} f_{36}^{2}}{f_{2}^{6} f_{12}^{2}}. \label{lemma5.2} 
\end{align}
\end{lemma}

\begin{proof}
Equation \eqref{lemma5.1} was proved by Xia and Yao \cite{XiaYao2012}. Equation \eqref{lemma5.2} follows by squaring both sides of \eqref{lemma5.1}.
\end{proof}

\begin{lemma}\cite[Eq. (3.1)]{Toh2012} \label{lemma6} The following $3$-dissection holds.
\begin{equation}
\dfrac{f_{2}^{3}}{f_{1}^{3}}=\dfrac{f_{6}}{f_{3}} + 3q\dfrac{f_{6}^{4}f_{9}^{5}}{f_{3}^{8}f_{18}} +6q^{2} \dfrac{f_{6}^{3}f_{9}^{2}f_{18}^{2}}{f_{3}^{7}} +12q^{3} \dfrac{f_{6}^{2}f_{18}^{5}}{f_{3}^{6}f_{9}}. \label{lemma6.1}
\end{equation}
\end{lemma}

\begin{lemma} \label{lemma7} The following formulas hold.
\begin{align}
\dfrac{f_{2}}{f_{1}^{2}}&=\dfrac{f_{6}^{4}f_{9}^{6}}{f_{3}^{8}f_{18}^{3}}+2q\dfrac{f_{6}^{3}f_{9}^{3}}{f_{3}^{7}}+4q^{2}\dfrac{f_{6}^{2}f_{18}^{3}}{f_{3}^{6}},\label{lemma7.1}\\
&= \varphi(q) \varphi(q^2)^2 \varphi(q^4)^4 \varphi(q^8)^8 \cdots. \label{lemma7.2}
\end{align}
\end{lemma}

\begin{proof}
Equation \eqref{lemma7.1} was proved by Hirschhorn and Sellers \cite[Eq. (3), Thm. 1]{HirschhornSellers2005b}. Equation \eqref{lemma7.2} is Equation (1.5.16) in \cite{Hirschhorn2017}.
\end{proof}

\begin{lemma} \label{lemma8}  The following $3$-dissections hold.
\begin{align}
\dfrac{f_{2}^{2}}{f_{1}}&=\dfrac{f_{6}f_{9}^{2}}{f_{3}f_{18}}+q\dfrac{f_{18}^{2}}{f_{9}}, \label{lemma8.1}\\
\psi(-q)&= A(q^{3}) - q \psi(-q^{9}), \label{lemma8.2}\\
&= P(-q^{3}) - q \psi(-q^{9}), \label{lemma8.3}
\end{align}
where 
\[
A(q^{3})= \dfrac{f_{3} f_{12} f_{18}^{5}}{f_{6}^{2} f_{9}^{2} f_{36}^{2}} \ \mathrm{and} \ P(-q)=\sum_{n \geq 0}(-q)^{(3n^2 \pm n)/2} =  \frac{f_1 f_4 f_6^5}{f_2^2 f_3^2 f_{12}^2}.
\]
\end{lemma}
\begin{proof}
Equation \eqref{lemma8.1} is Corollary (ii) of \cite[p. 49]{Berndt1991}. Changing $q$ by $-q$ in Corollary (ii) of \cite[p. 49]{Berndt1991}, and using $(-q;-q)_{\infty}=f_{2}^{3}/f_{1} f_{4}$, we obtain \eqref{lemma8.2}. The proof of Equation \eqref{lemma8.3} can be found in \cite[p. 274-275]{HirschhornSellers2010}
\end{proof}

\begin{lemma} \label{lemma10} The following $3$-dissection holds.
\begin{equation}
\varphi(-q)=\varphi(-q^9) -2q\dfrac{f_{3}f_{18}^{2}}{f_{6}f_{9}}. \label{lemma10.1}
\end{equation}
\end{lemma}
\noindent Identity \eqref{lemma10.1} is equivalent to the three-dissection of $\phi(-q)$ \cite{AndrewsHirschhornSellers2010}.

\begin{lemma}\cite[Eq. 5.1]{Son1998} We have \label{lemma9} 
\begin{equation} 
\frac{f_{2}^{4} f_{3}^{8}}{f_{1}^{8} f_{6}^{4}} = 1+8q \frac{f_{2} f_{6}^{5}}{f_{1}^{5} f_{3}}. \label{lemma9.1}
\end{equation}
\end{lemma}

Finally, we recall the Legendre symbol. Let $p$ be any odd prime and $\delta$ be any integer relatively prime to $p$. The \emph{Legendre symbol} $\bigl( \frac{\delta}{p} \bigl)$ is defined by
\begin{equation*}
        \left( \frac{\delta}{p} \right) = \begin{cases}
                      1,  &\text{if $\delta$ is a quadratic residue of $p$,}\\
                     -1,  &\text{if $\delta$ is a quadratic non-residue of $p$.}\\
                    \end{cases}
\end{equation*}

\subsection{Results from the theory of modular forms}

We define the matrix groups
	\begin{align*}
		\mathrm{SL_2}(\mathbb{Z}):= &\left\{ \begin{bmatrix}
			a && b \\c && d
		\end{bmatrix}: a, b, c, d \in \mathbb{Z}, ad-bc=1 \right\},\\
		\Gamma_{\infty}:= &\left\{\begin{bmatrix}
			1 &n\\ 0&1	\end{bmatrix}: n \in \mathbb{Z}\right\}.
	\end{align*}
	For a positive integer $N$, we define
	\begin{align*}
		\Gamma_{0}(N):=& \left\{ \begin{bmatrix}
			a && b \\c && d
		\end{bmatrix} \in \mathrm{SL_2}(\mathbb{Z}) : c\equiv0 \pmod N \right\},\\
		\Gamma_{1}(N):=& \left\{ \begin{bmatrix}
			a && b \\c && d
		\end{bmatrix} \in \Gamma_{0}(N) : a\equiv d  \equiv 1 \pmod N \right\}
	\end{align*}
	and 
	\begin{align*}
		\Gamma(N):= \left\{ \begin{bmatrix}
			a && b \\c && d
		\end{bmatrix} \in \mathrm{SL_2}(\mathbb{Z}) : a\equiv d  \equiv 1 \pmod N,  b \equiv c  \equiv 0 \pmod N \right\}.
	\end{align*}
	A subgroup of $\mathrm{SL_2}(\mathbb{Z})$ is called a congruence subgroup if it contains $ \Gamma(N)$ for some $N$ and the smallest $N$ with this property is called its level. Note that $ \Gamma_{0}(N)$ and $ \Gamma_{1}(N)$ are congruence subgroups of level $N,$ whereas $ \mathrm{SL_2}(\mathbb{Z}) $ and $\Gamma_{\infty}$ are congruence subgroups of level $1.$ The index of $\Gamma_0(N)$ in $\mathrm{SL_2}(\mathbb{Z})$ is 
	\begin{align*}
		[\mathrm{SL_2}(\mathbb{Z}):\Gamma_0(N)]=N\prod\limits_{p|N}\left(1+\frac 1p\right).
	\end{align*}

	Let $\mathbb{H}$ denote the upper half of the complex plane $\mathbb{C}$. The group 
	\begin{align*}
		\mathrm{GL_2^{+}}(\mathbb{R}):= \left\{ \begin{bmatrix}
			a && b \\c && d
		\end{bmatrix}: a, b, c, d \in \mathbb{R}, ad-bc>0 \right\},
	\end{align*}
	acts on $\mathbb{H}$ by $ \begin{bmatrix}
		a && b \\c && d
	\end{bmatrix} z = \dfrac{az+b}{cz+d}.$ We identify $\infty$ with $\dfrac{1}{0}$ and define $ \begin{bmatrix}
		a && b \\c && d
	\end{bmatrix} \dfrac{r}{s} = \dfrac{ar+bs}{cr+ds},$ where $\dfrac{r}{s} \in \mathbb{Q} \cup \{ \infty\}$. This gives an action of $\mathrm{GL_2^{+}}(\mathbb{R})$ on the extended half plane $\mathbb{H}^{*}=\mathbb{H} \cup \mathbb{Q} \cup \{\infty\}$. Suppose that $\Gamma$ is a congruence subgroup of $\mathrm{SL_2}(\mathbb{Z})$. A cusp of $\Gamma$ is an equivalence class in $\mathbb{P}^{1}=\mathbb{Q} \cup \{\infty\}$ under the action of $\Gamma$.
	
	The group $\mathrm{GL_2^{+}}(\mathbb{R})$ also acts on functions $g:\mathbb{H} \rightarrow \mathbb{C}$. In particular, suppose that $\gamma=\begin{bmatrix}
		a && b \\c && d
	\end{bmatrix}\in \mathrm{GL_2^{+}}(\mathbb{R})$. If $f(z)$ is a meromorphic function on $\mathbb{H}$ and $k$ is an integer, then define the slash operator $|_{k}$ by
	\begin{align*}
		(f|_{k} \gamma)(z):= (\det \gamma)^{k/2} (cz+d)^{-k} f(\gamma z).
	\end{align*}
	
	\begin{definition}
		Let $\Gamma$ be a congruence subgroup of level $N$. A holomorphic function $f:\mathbb{H} \rightarrow \mathbb{C}$ is called a modular form with integer weight $k$ on $\Gamma$ if the following hold:
		\begin{enumerate}[$(1)$]
			\item We have
			\begin{align*}
				f \left( \frac{az+b}{cz+d}\right)=(cz+d)^{k} f(z)
			\end{align*}
			for all $z \in \mathbb{H}$ and $\begin{bmatrix}
				a && b \\c && d
			\end{bmatrix}\in \Gamma$. 
			\item If $\gamma\in SL_2 (\mathbb{Z})$, then $(f|_{k} \gamma)(z)$ has a Fourier expnasion of the form
			\begin{align*}
				(f|_{k} \gamma)(z):= \sum \limits_{n\geq 0}a_{\gamma}(n) q_N^{n}
			\end{align*}
			where $q_N:=e^{2\pi i z /N}$.
		\end{enumerate}
	\end{definition}
	For a positive integer $k$, the complex vector space of modular forms of weight $k$ with respect to a congruence subgroup $\Gamma$ is denoted by $M_{k}(\Gamma)$.
	
	\begin{definition} \cite[Definition 1.15]{Ono2004}
		If $\chi$ is a Dirichlet character modulo $N$, then we say that a modular form $g \in M_{k}(\Gamma_1(N))$ has Nebentypus character $\chi$ if 
		\begin{align*}
			f \left( \frac{az+b}{cz+d}\right)=\chi(d) (cz+d)^{k} f(z)
		\end{align*}
		for all $z \in \mathbb{H}$ and $\begin{bmatrix}
			a && b \\c && d
		\end{bmatrix}\in \Gamma_{0}(N)$. The space of such modular forms is denoted by $M_{k}(\Gamma_0(N), \chi)$.
	\end{definition}
	
	The relevant modular forms for the results obtained in this article arise from eta-quotients. Recall that the Dedekind eta-function $\eta (z)$ is defined by 
	\begin{align*}
		\eta (z):= q^{1/24}(q;q)_{\infty},
	\end{align*}
	where $q:=e^{2\pi i z}$ and $z \in \mathbb{H}$. A function $f(z)$ is called an eta-quotient if it is of the form
	\begin{align*}
		f(z):= \prod\limits_{\delta|N} \eta(\delta z)^{r_{\delta}}
	\end{align*}
	where $N$ and $r_{\delta}$ are integers with $N>0$. 
	
	\begin{theorem} \cite[Theorem 1.64]{Ono2004} \label{thm2.3}
		If $f(z)=\prod\limits_{\delta|N} \eta(\delta z)^{r_{\delta}}$ is an eta-quotient such that $k= \frac 12$ $\sum_{\delta|N} r_{\delta}\in \mathbb{Z}$, 
		\begin{align*}
			\sum\limits_{\delta|N} \delta r_{\delta} \equiv 0\pmod {24}	\quad \textrm{and} \quad \sum\limits_{\delta|N} \frac{N}{\delta}r_{\delta} \equiv 0\pmod {24},
		\end{align*}
		then $f(z)$ satisfies
		\begin{align*}
			f \left( \frac{az+b}{cz+d}\right)=\chi(d) (cz+d)^{k} f(z)
		\end{align*}
		for each $\begin{bmatrix}
			a && b \\c && d
		\end{bmatrix}\in \Gamma_{0}(N)$. Here the character $\chi$ is defined by $\chi(d):= \left(\frac{(-1)^{k}s}{d}\right)$ where $s=\prod_{\delta|N} \delta ^{r_{\delta}}$.
	\end{theorem}
	
	\begin{theorem} \cite[Theorem 1.65]{Ono2004} \label{thm2.4}
		Let $c,d$ and $N$ be positive integers with $d|N$ and $\gcd(c,d)=1$. If $f$ is an eta-quotient satisfying the conditions of Theorem \ref{thm2.3} for $N$, then the order of vanishing of $f(z)$ at the cusp $\frac{c}{d}$ is
		\begin{align*}
			\frac{N}{24}\sum\limits_{\delta|N} \frac{\gcd(d, \delta)^2 r_{\delta}}{\gcd(d, \frac{N}{ d} )d \delta}.
		\end{align*}
	\end{theorem}
	Suppose that $f(z)$ is an eta-quotient satisfying the conditions of Theorem \ref{thm2.3} and that the associated weight $k$ is a positive integer. If $f(z)$ is holomorphic at all of the cusps of $\Gamma_0(N)$, then $f(z) \in M_{k}(\Gamma_0(N), \chi)$. Theorem \ref{thm2.4} gives the necessary criterion for determining orders of an eta-quotient at cusps. In the proofs of our results, we use Theorems \ref{thm2.3} and \ref{thm2.4} to prove that $f(z) \in M_{k}(\Gamma_0(N), \chi)$ for certain eta-quotients,$f(z),$ we consider in the sequel.

	We finally recall the definition of Hecke operators and a few relevant results. Let $m$ be a positive integer and $f(z)= \sum \limits_{n= 0}^{\infty}b(n) q^{n}\in M_{k}(\Gamma_0(N), \chi)$. Then the action of Hecke operator $T_m$ on $f(z)$ is defined by
	\begin{align*}
		f(z)|T_{m} := \sum \limits_{n= 0}^{\infty} \left(\sum \limits_{d|\gcd(n,m)} \chi(d) d^{k-1} b\left(\frac{mn}{d^2}\right)\right)q^{n}.
	\end{align*}
	In particular, if $m=p$ is a prime, we have
	\begin{align*}
		f(z)|T_p := \sum \limits_{n= 0}^{\infty}\left( b(pn) + \chi(p) p^{k-1} b\left(\frac{n}{p}\right)\right)q^{n}.
	\end{align*}
	We note that $b(n)=0$ unless $n$ is a nonnegative integer.

\section{Congruence for $\overline{S}_{t}(n)$ modulo $4$} \label{sec1}
In this section, we derive a Ramanujan-like congruence modulo $4$ for the number $\overline{S}_{t}(n)$, valid for any odd $t\geq 3$, by characterizing the values of $n$ for which they hold.

\begin{theorem} For any odd positive integer $t\geq 3$ and for all $n \geq 0$, where $t$ is not a perfect square, we have
\[
 \overline{S}_{t}(n) \equiv
  \begin{cases}
    1 \ (\mathrm{mod} \ 4), & \text{if} \ n=0; \\
    2 \ (\mathrm{mod} \ 4), & \text{if} \ n=j^{2}, 2j^{2}, tj^{2}, 2tj^{2} \ \text{for some j}; \\
    0 \ (\mathrm{mod} \ 4), & \text{otherwise.}
  \end{cases}
\]
If $t$ is a perfect square, then we have
\[
 \overline{S}_{t}(n) \equiv
  \begin{cases}
    1 \ (\mathrm{mod} \ 4), & \text{if} \ n=0; \\
    2 \ (\mathrm{mod} \ 4), & \text{if} \ n=j^{2}, 2j^{2}, 2tj^{2} \ \text{for $j>0$, where if} \ n=j^2 \ \text{then} \ n\not\equiv 0  \ (\mathrm{mod} \ t); \\
    0 \ (\mathrm{mod} \ 4), & \text{otherwise.}
  \end{cases}
\]
\end{theorem}

\begin{proof} For $t \geq 3$ odd, we have that
\begin{align}
\sum_{n\geq 0} \overline{S}_{t}(n) q^{n} &= \bigg( \frac{f_{2}}{f_{1}^{2}} \bigg) \bigg( \frac{f_{2}^{2}}{f_{4}} \bigg) \bigg( \frac{f_{t}^{2}}{f_{2t}} \bigg) \bigg( \frac{f_{4t}}{f_{2t}^{2}} \bigg)\\
&= \bigg( \varphi(q) \varphi(q^2)^2 \varphi(q^4)^4 \cdots \bigg) \varphi(-q^2) \varphi(-q^t) \bigg( \varphi(q^{2t}) \varphi(q^{4t})^2 \varphi(q^{8t})^4 \cdots \bigg). \label{maineqexpansion}
\end{align}
In view of \eqref{binomiallemma} and \eqref{maineqexpansion}, with $p=k=2$, we find that
\[
\sum_{n\geq 0} \overline{S}_{t}(n) q^{n} \equiv \varphi(q) \varphi(q^{2}) \varphi(q^{t}) \varphi(q^{2t}) \ (\mathrm{mod} \ 4) .
\]
That is,
\[
\sum_{n\geq 0} \overline{S}_{t}(n) q^{n} \equiv  \bigg( 1+2\sum_{n=1}^{\infty} q^{n^{2}} \bigg) \bigg( 1+2\sum_{n=1}^{\infty} q^{2n^{2}} \bigg) \bigg( 1+2\sum_{n=1}^{\infty} q^{tn^{2}} \bigg) \bigg( 1+2\sum_{n=1}^{\infty} q^{2tn^{2}} \bigg) \ (\mathrm{mod} \ 4) .
\]
By further expanding the above equation, we get
\[
\sum_{n\geq 0} \overline{S}_{t}(n) q^{n} \equiv  1+ 2\sum_{n=1}^{\infty} q^{n^{2}} +2\sum_{n=1}^{\infty} q^{2n^{2}} + 2\sum_{n=1}^{\infty} q^{tn^{2}}+2\sum_{n=1}^{\infty} q^{2tn^{2}}  \ (\mathrm{mod} \ 4) ,
\]
Now, it $t$ is a perfect square, then some of the terms generated by $q^{n^{2}}$ and $q^{tn^{2}}$ will appear twice with coefficients equal to $4$, from which our result follows.
\end{proof}

\section{Congruences for $\overline{S}_{t}(n)$ when $t=2^k$ and $t=3^k$} \label{sec4}
In this section, we present several Ramanujan-like congruence for the generating function \eqref{maineq} when $t=2^k$ for $k\geq 3$, and when $t=3^k$ for $k\geq 2$. These two results motivates us to look for further specific congruences for the cases $t=3$ and $t=9$ in the later sections.

\begin{theorem}
    For all $n\geq 0, k\geq 3$, we have
    \begin{align*}
        \sst_{2^k}(8n+1)&\equiv 0\pmod{2},\\
    \sst_{2^k}(8n+2)&\equiv 0\pmod{4},\\
    \sst_{2^k}(8n+3)&\equiv 0\pmod{8},\\
    \sst_{2^k}(8n+4)&\equiv 0\pmod{2},\\
    \sst_{2^k}(8n+5)&\equiv 0\pmod{8},\\
    \sst_{2^k}(8n+6)&\equiv 0\pmod{8},\\
    \sst_{2^k}(8n+7)&\equiv 0\pmod{32}.
    \end{align*}
\end{theorem}

\begin{proof}
Using \eqref{maineq} and the property
\begin{equation*}
    \frac{1}{\varphi(-q)}=\varphi(q)\varphi(q^2)^2\varphi(q^4)^4\cdots =\prod_{i\geq 0}\varphi(q^{2^i})^{2^i},
\end{equation*}
rewrite
\[
\sum_{n\geq 0}\sst_{2^i}(n)q^n=\varphi(q)\varphi(q^2)\varphi(q^4)^2\prod_{i\geq 3}\varphi(q^{2^i})^{2^{i-1}}\frac{\varphi(-q^{2^k})}{\varphi(-q^{2^{i+1}})}.
\]
Since $\prod_{i\geq 3}\varphi(q^{2^i})^{2^i-1}\frac{\varphi(-q^{2^k})}{\varphi(-q^{2^{i+1}})}$ is a function of $q^8$ for all $k\geq 3$, to prove our result we can ignore that part and perform a $8$-dissection of the remaining part.

As the highest modulus involved in the theorem is $32$ and the other moduli are divisors of $32$, we will prove our result if we consider this $8$-dissection modulo $32$. Rewrite
\[
\sum_{n\geq 0}\sst_{2^i}(n)q^n=\left(\sum_{j=0}^{7}a_{t,j}q^{j}F_{t,j}(q^8)\right) \left(\prod_{i\geq 3}\varphi(q^{2^i})^{2^i-1}\frac{\varphi(-q^{2^k})}{\varphi(-q^{2^{i+1}})}\right),
\]
where $F_{t,j}(q^8)$ is a function of $q^8$ whose power series representation has integer coefficients. It suffices to just prove the following congruences
\begin{align*}
    a_{t,1} &\equiv 0 \pmod{2},\\
    a_{t,2} &\equiv 0 \pmod{4},\\
    a_{t,3} &\equiv 0 \pmod{8},\\
    a_{t,4} &\equiv 0 \pmod{2},\\
    a_{t,5} &\equiv 0 \pmod{8},\\
    a_{t,6} &\equiv 0 \pmod{8},\\
    a_{t,7} &\equiv 0 \pmod{32}.
\end{align*}
However, this already follows from a result of Saikia,  Sarma, and Sellers \cite[Lemma 6.1]{SSS25}, who follow the same pattern as the proof of \cite[Theorem 2.2]{Sel24}, so we omit the details here. The interested reader can refer to the above-mentioned papers, as well as \cite[Theorem 1.9]{SS25} to fill in the details.
\end{proof}

\begin{theorem}\label{thmhs}
For all $\alpha\geq 0$, $k\geq 2$, and $n\geq 0$, we have
\begin{equation}\label{eq3j}
\sst_{3^k}(9^\alpha(9n+6))\equiv 0 \pmod 3.
\end{equation}
\end{theorem}

\begin{proof}
We rewrite the generating function for $\sst_{3^k}(n)$ in the following fashion
\begin{equation}\label{eqqq}
\sum_{n\geq 0}\sst_{3^k}(n)q^n=\left(\sum_{n\geq 0}\bar p_o(n)q^n\right)\frac{f_{3^k}^2f_{4\cdot 3^k}}{f^3_{2\cdot 3^k}},
\end{equation}
where $\bar p_o(n)$ is the number of overpartitions of $n$ into odd parts, and its generating function is given by
\[
\sum_{n\geq 0}\bar p_o(n)q^n=\frac{f_2^3}{f_1^2f_4}.
\]
We need the following identities of Hirschhorn and Sellers \cite{HS06}
\begin{align}
\sum_{n\geq 0}\bar p_o(9n+6)q^n &\equiv 0 \pmod 3,\label{hs1}\\
\sum_{n\geq 0}\bar p_o(27n)q^n &\equiv \sum_{n\geq 0}\bar p_o(3n)q^n \pmod 3\label{hs2}.
\end{align}
Identity \eqref{hs1} is \cite[Theorem 3.5]{HS06} and identity \eqref{hs2} is \cite[Theorem 3.8]{HS06}.

We notice that, we can rewrite equation \eqref{eqqq} as
\[
\sum_{n\geq 0}\sst_{3^k}(n)q^n=\left(\sum_{n\geq 0}\bar p_o(3n)q^{3n}+\sum_{n\geq 0}\bar p_o(3n+1)q^{3n+1}+\sum_{n\geq 0}\bar p_o(3n+2)q^{3n+2}\right)\frac{f_{3^k}^2f_{4\cdot 3^k}}{f^3_{2\cdot 3^k}},
\]
from which we obtain (for $k\geq 1$) by extracting the terms involving the powers of $q^3$ and then doing the transformation $q^3\rightarrow q$,
\begin{equation}\label{hs3}
\sum_{n\geq 0}\sst_{3^k}(3n)q^n=\left(\sum_{n\geq 0}\bar p_o(3n)q^{n}\right)\frac{f_{3^{k-1}}^2f_{4\cdot 3^{k-1}}}{f^3_{2\cdot 3^{k-1}}}.
\end{equation}
In a similar fashion, we obtain, for $k\geq 2$
\begin{equation}\label{hs4}
\sum_{n\geq 0}\sst_{3^k}(9n+6)q^n=\left(\sum_{n\geq 0}\bar p_o(9n+6)q^{n}\right)\frac{f_{3^{k-2}}^2f_{4\cdot 3^{k-2}}}{f^3_{2\cdot 3^{k-2}}}.
\end{equation}
Equation \eqref{hs1} and \eqref{hs4} proves, for all $k\geq 2$ and $n\geq 0$, we have
\[
\sst_{3^k}(9n+6)\equiv 0 \pmod 3.
\]
Using this congruence and equations \eqref{hs2}, \eqref{hs3}, and \eqref{hs4} we complete the proof of our result.
\end{proof}

\section{Congruences for $\overline{S}_{3}(n)$} \label{sec2}
In this section, we derive several Ramanujan-like and families of congruences for $\overline{S}_{3}(n)$ modulo $3,6,8$ and $12$.

\begin{theorem} \label{thm1} For any nonnegative integer $n$, we have
\begin{align}
\overline{S}_{3}(9n+6) &\equiv0 \ (\mathrm{mod} \ 3), \label{thm1-cong11}\\
\overline{S}_{3}(24n) &\equiv  0 \ (\mathrm{mod} \ 3) \ \text{for} \ n>0, \label{thm1-cong1}\\
\overline{S}_{3}(24n+4) &\equiv  2 f_{1}^{4} \ (\mathrm{mod} \ 3), \label{thm1-cong2}\\
\overline{S}_{3}(24n+12) &\equiv  \psi(q) \psi(q^{3}) \ (\mathrm{mod} \ 3), \label{thm1-cong3}\\
\overline{S}_{3}(24n+16) &\equiv  0 \ (\mathrm{mod} \ 3), \label{thm1-cong4}\\
\overline{S}_{3}(27n) &\equiv  \overline{S}_{3}(3n) \ (\mathrm{mod} \ 3). \label{thm1-cong5}
\end{align}
\end{theorem}

\begin{proof} Setting $t=3$ in \eqref{maineq} and then substituting \eqref{lemma1.1} into the resulting equation, we find that
\[
\sum_{n \geq 0} \overline{S}_{3}(n) q^{n} = \dfrac{f_{4}^{3} f_{12}^{3}}{f_{2}^{2} f_{6}^{2} f_{8} f_{24}} + 2q \dfrac{f_{8} f_{24}}{f_{2} f_{6}}.
\]
Equating the even and odd powers on both sides of the above equation, we get
\begin{align}
&\sum_{n \geq 0} \overline{S}_{3}(2n) q^{n} = \dfrac{f_{2}^{3} f_{6}^{3}}{f_{1}^{2} f_{3}^{2} f_{4} f_{12}}, \label{a1}\\
&\sum_{n \geq 0} \overline{S}_{3}(2n+1) q^{n} = 2\dfrac{f_{4} f_{12}}{f_{1} f_{3}}. \label{a2}
\end{align}
Applying Lemma \ref{binomial} to \eqref{a1}, with $p=3$ and $k=1$, we obtain
\begin{equation}
\sum_{n \geq 0} \overline{S}_{3}(2n) q^{n} \equiv  \dfrac{f_{1} f_{2}^{3} f_{6}^{3}} {f_{3}^{3} f_{4} f_{12}} \ (\mathrm{mod} \ 3). \label{a3}
\end{equation}
Using \eqref{lemma1.3} in \eqref{a3} and then extracting the terms of the form $q^{2n}$ from both sides of the resulting equation, we arrive at
\begin{equation}
\sum_{n \geq 0} \overline{S}_{3}(4n) q^{n} \equiv  \dfrac{f_{1}^{4} f_{2} f_{6}} {f_{3}^{4}} \equiv \dfrac{f_{1} f_{2} f_{6}} {f_{3}^{3}}  \ (\mathrm{mod} \ 3). \label{a4}
\end{equation}
Substituting \eqref{lemma2.1} into \eqref{a4}, we find that
\begin{equation}
\sum_{n \geq 0} \overline{S}_{3}(4n) q^{n} \equiv  \dfrac{f_{6}^{2} f_{9}^{4}} {f_{3}^{4} f_{18}^{2}} - q \dfrac{f_{6} f_ {9} f_{18}}{f_{3}^{3}} -2q^{2} \dfrac{f_{18}^{4}}{f_{3}^{2} f_{9}^{2}}  \ (\mathrm{mod} \ 3). \label{a5}
\end{equation}
Extracting the terms of the form $q^{3n+j}$ for $j=0,1$  from both sides of equation \eqref{a5}, we arrive at
\begin{align}
&\sum_{n \geq 0} \overline{S}_{3}(12n) q^{n} \equiv  \dfrac{f_{2}^{2} f_{3}^{4}} {f_{1}^{4} f_{6}^{2}} \equiv \dfrac{f_{2}^{2} f_{3}^{3}} {f_{1} f_{6}^{2}}   \ (\mathrm{mod} \ 3), \label{a6}\\
&\sum_{n \geq 0} \overline{S}_{3}(12n+4) q^{n} \equiv  2 \dfrac{f_{2} f_ {3} f_{6}}{f_{1}^{3}}  \ (\mathrm{mod} \ 3). \label{a7}
\end{align}
Employing \eqref{lemma1.2} into \eqref{a6} and then extracting the even and the odd powers from both sides of the resulting equation, we find that
\begin{align}
&\sum_{n \geq 0} \overline{S}_{3}(24n) q^{n} \equiv  \dfrac{f_{2}^{3}} {f_{6}} \equiv 1 \ (\mathrm{mod} \ 3), \label{a8}\\
&\sum_{n \geq 0} \overline{S}_{3}(24n+12) q^{n} \equiv   \dfrac{f_{1}^{2} f_ {6}^{3}}{f_{2} f_  {3}^{2}} \equiv \dfrac{f_{2}^{2} f_ {6}^{2}}{f_{1} f_{3}}  \ (\mathrm{mod} \ 3). \label{a9}
\end{align}
Congruence \eqref{thm1-cong1} follows from \eqref{a8} and congruence \eqref{thm1-cong3} follows from \eqref{a9} by using the definition of $\psi(q)$.

Substituting \eqref{lemma1.4} into \eqref{a7}, we get
\begin{equation}
\sum_{n \geq 0} \overline{S}_{3}(12n+4) q^{n} \equiv  2 \dfrac{f_{4}^{6} f_{6}^{4}}{f_{2}^{8} f_{12}^{2}} + 6 q^{2}  \dfrac{f_{4}^{2} f_{6}^{2} f_{12}^{2}}{f_{2}^{6}}  \ (\mathrm{mod} \ 3).  \label{a10}
\end{equation}
Congruence \eqref{thm1-cong4} follows from the above equation by extracting the odd powers from both sides. Equating the even powers from both sides of equation \eqref{a10}, we find that
\begin{equation}
\sum_{n \geq 0} \overline{S}_{3}(24n+4) q^{n} \equiv  2 \dfrac{f_{2}^{6} f_{3}^{4}}{f_{1}^{8} f_{6}^{2}} \equiv 2f_{1}^{4}  \ (\mathrm{mod} \ 3).  \label{a11}
\end{equation}
Congruence \eqref{thm1-cong2} is an immediate result from \eqref{a11}.

Again, setting $t=3$ in \eqref{maineq} and applying \eqref{binomiallemma}, with $p=3$ and $k=1$, we get 
\begin{equation*}
\sum_{n \geq 0} \overline{S}_{3}(n) q^{n} = \frac{f_{2}^3 }{f_{1}^2 f_{4}} \cdot \frac{f_{3}^2 f_{12}}{f_{6}^3} \equiv \frac{f_{1} f_{3} f_{12}}{f_{4} f_{6}^2} \equiv \frac{f_{1}}{f_{4}} \cdot  \frac{f_{3} f_{12}}{f_{6}^2} \ (\mathrm{mod} \ 3).
\end{equation*}
Now by invoking \eqref{lemma4.2} into the above equation and then extracting the terms of the form $q^{3n}$ from both sides of the resulting equation, we arrive at
\begin{equation*}
\sum_{n \geq 0} \overline{S}_{3}(3n) q^{n} \equiv \frac{f_1 f_3 f_6}{f_2 f_4^2} \equiv \frac{f_1 f_3 f_4 f_6}{f_2 f_{12}} = \psi(-q) \frac{f_3 f_6}{f_{12}} \ (\mathrm{mod} \ 3).
\end{equation*}
By using \eqref{lemma8.3} in the above equation, we see that
\begin{equation}
\sum_{n \geq 0} \overline{S}_{3}(3n) q^{n} \equiv\big( P(-q^3) - q \psi(-q^9) \big) \frac{f_3 f_6}{f_{12}} \ (\mathrm{mod} \ 3),  \label{sel01}
\end{equation}
and since the above $3$-dissection contains no terms of the form $q^{3n+2}$, we deduce that
\begin{equation*}
\overline{S}_{3}(9n+6)  \equiv 0 \ (\mathrm{mod} \ 3).
\end{equation*}

From \eqref{sel01}, we know that
\begin{align*}
\sum_{n \geq 0} \overline{S}_{3}(9n) q^{n} &\equiv P(-q) \frac{f_1 f_2}{f_{4}} = \frac{f_1 f_4 f_6^5}{f_2^2 f_3^2 f_{12}^2}  \cdot \frac{f_1 f_2}{f_{4}} = \frac{f_1^2 f_6^5}{f_2 f_3^2 f_{12}^2} \\
&\equiv  \varphi(-q) \bigg( \frac{ f_6^5}{ f_3^2 f_{12}^2} \bigg)
\ (\mathrm{mod} \ 3).
\end{align*}
By using \eqref{lemma10.1} in the above equation, we arrive at
\[
\sum_{n \geq 0} \overline{S}_{3}(9n) q^{n} \equiv \bigg(  \varphi(-q^9) -2q\dfrac{f_{3}f_{18}^{2}}{f_{6}f_{9}} \bigg) \bigg( \frac{ f_6^5}{ f_3^2 f_{12}^2} \bigg) \ (\mathrm{mod} \ 3).
\]
Thus, 
\[
\sum_{n \geq 0} \overline{S}_{3}(27n) q^{n} \equiv  \varphi(-q^3) \frac{ f_2^5}{ f_1^2 f_{4}^2} =  \frac{ f_2^5 f_3^2}{ f_1^2 f_{4}^2 f_6} \equiv \frac{f_1 f_3 f_6}{f_2 f_4^2} \equiv \sum_{n \geq 0} \overline{S}_{3}(3n) q^{n}  \ (\mathrm{mod} \ 3).
\]
Therefore, for all $n\geq 0$, we have
\[
\overline{S}_{3}(27n) \equiv  \overline{S}_{3}(3n) \ (\mathrm{mod} \ 3).
\]
\end{proof}

\begin{theorem} \label{thm3}  For all nonnegative integers $\alpha$ and  $n$, we have
\begin{equation*}
\overline{S}_{3} \bigl( 24\cdot 5^{2\alpha+2}n+24\cdot 5^{2\alpha+1}i +  4 \cdot 5^{2\alpha+2} \bigl) \equiv 0 \ (\mathrm{mod} \ 3), 
\end{equation*}
for all $1\leq i \leq 4$.
\end{theorem}

\begin{proof} Substituting Equation \eqref{lemma3.1} into \eqref{thm1-cong2} and extracting the terms involving $q^{5n+4}$, we obtain
\begin{equation}
\sum_{n \geq 0} \overline{S}_{3}(120n+100) q^{n} \equiv 2f_{5}^{4} \ (\mathrm{mod} \ 3), \label{f1}
\end{equation}
which implies,
\begin{equation}
\sum_{n \geq 0} \overline{S}_{3}(600n+100) q^{n} \equiv 2f_{1}^{4} \ (\mathrm{mod} \ 3). \label{f2}
\end{equation}
From Equations \eqref{thm1-cong2} and \eqref{f2}, we find that
\begin{equation}
\overline{S}_{3}(24n+4) \equiv \overline{S}_{3}(600n+100) \ (\mathrm{mod} \ 3). \label{f3}
\end{equation}
From  \eqref{f3} and by mathematical induction on $\alpha \geq 0$, we obtain
\begin{equation}
\overline{S}_{3}(24n+4) \equiv \overline{S}_{3}(24 \cdot 5^{2\alpha+2}n+4\cdot 5^{2\alpha+2}) \ (\mathrm{mod} \ 3). \label{f4}
\end{equation}
From \eqref{f1}, we have
\begin{equation}
\overline{S}_{3}(600n+120i+100) \equiv 0 \ (\mathrm{mod} \ 3), \quad  i=1,2,3,4. \label{f5}
\end{equation}
Using \eqref{f4} and \eqref{f5}, we obtain the desired result.
\end{proof}

\begin{theorem} \label{thm2} For any prime $p\equiv 5 \ (\mathrm{mod} \ 6)$, $\alpha \geq 0$, and $n\geq 0$, we have
\begin{equation}
\overline{S}_{3} \left(24p^{2\alpha+1}(pn+i)+12p^{2\alpha+2}\right)\equiv 0 \ (\mathrm{mod} \ 3), \label{s3-famcongmod3}
\end{equation}
for all $1\leq i \leq p-1$.
\end{theorem}

\begin{proof} Define
\begin{equation} 
\sum_{n\geq 0} a(n)q^{n} = \psi(q) \psi(q^{3}). \label{d1}
\end{equation}
From  Equation \eqref{thm1-cong3} we have 
\begin{equation}
\overline{S}_{3}(24n+12) \equiv a(n) \ (\mathrm{mod} \ 3).  \label{d2}
\end{equation} 
Now, consider the congruence equation
\[
\frac{k^{2}+k}{2}+3\cdot \frac{m^{2}+m}{2} \equiv  \frac{4p^{2}-4}{8} \ (\mathrm{mod} \ p), 
\]
which is equivalent to
\begin{equation}
(2k+1)^{2}+ 3\cdot (2m+1)^{2}  \equiv  0 \ (\mathrm{mod} \ p), \label{d3}
\end{equation}
where $0 \leq k,m \leq (p-1)/2$ and $p$ is a prime number such that ($\frac{-3}{p}$)$=-1$. Since ($\frac{-3}{p}$)$=-1$ for $p \equiv 5 \ (\mathrm{mod} \ 6)$, the congruence relation of Equation \eqref{d3} holds if and only if both $k=m=(p-1)/2$. Substitute Equation \eqref{psidissection} into \eqref{d3} and extract the terms in which the powers of $q$ are congruent to $(p^{2}-1)/12$ modulo $p$, and then divide by $q^{\frac{p^{2}-1}{12}}$, we find that
\[
\sum_{n\geq 0} a \left(pn+\frac{p^{2}-1}{2} \right) q^{pn} = \psi(q^{p^{2}}) \psi(q^{3p^{2}}),
\]
which implies that 
\begin{equation}
\sum_{n\geq 0} a\left( p^{2}n+\frac{p^{2}-1}{2} \right)q^{n} = \psi(q) \psi(q^{3}), \label{d4}
\end{equation}
and for $n\geq 0$,
\begin{equation}
a \left(p^{2}n+pi+\frac{p^{2}-1}{2}\right) = 0, \label{d5}
\end{equation}
where  $1\leq i \leq p-1$. By induction, we obtain  that for all $n, \alpha \geq 0$,
\begin{equation}
a\left( p^{2\alpha}n+\frac{p^{2\alpha}-1}{2} \right)=b(n). \label{d6}
\end{equation}
Replacing $n$ by $p^{2}n+pi+\frac{p^{2}-1}{2}$ ($1\leq i \leq p-1$) in \eqref{d6} and using \eqref{d5}, we deduce that for $n\geq 0$ and $\alpha \geq 0$,
\[
a\left( p^{2\alpha+2}n+p^{2\alpha+1}i+\frac{p^{2\alpha+2}-1}{2} \right)=0.
\]
Replacing $n$ by $p^{2\alpha+2}n+p^{2\alpha+1}i+\frac{p^{2\alpha+2}-1}{2}$ in Equation \eqref{d2}, we obtain \eqref{s3-famcongmod3}.
\end{proof}

\begin{theorem} \label{thm4} For all integers  $n\geq 0$, we have
\begin{align}
\sum_{n\geq 0} \overline{S}_{3}(12n+2) q^{n} &\equiv  2f_{1}^{4} \ (\mathrm{mod} \ 8), \label{thm4-cong1}\\
\overline{S}_{3}(12n+6) &\equiv  0 \ (\mathrm{mod} \ 6), \label{thm4-cong2}\\
\overline{S}_{3}(12n+10) &\equiv  0 \ (\mathrm{mod} \ 12). \label{thm4-cong3}
\end{align}
\end{theorem}

\begin{proof} Substituting \eqref{lemma1.5} into \eqref{a1} then extracting the even and odd powers from both sides of the resulting equation, we arrive at
\begin{align}
&\sum_{n \geq 0} \overline{S}_{3}(4n) q^{n} = \dfrac{f_{4}^{5} f_{12}^{5}}{f_{1}^{2} f_{2} f_{3}^{2} f_{6} f_{8}^{2} f_{24}^{2}} + 4q^{2} \frac{f_{2} f_{6} f_{8}^{4} f_{24}^{2}}{f_{1}^{2} f_{3}^{2} f_{4} f_{12}}, \label{g01}\\
&\sum_{n \geq 0} \overline{S}_{3}(4n+2) q^{n} = 2\dfrac{f_{2}^{3} f_{6}^{3}}{f_{1}^{3} f_{3}^{3}}. \label{g1}
\end{align}
Again, using \eqref{lemma6.1} into \eqref{g1}, we find that
\[
\sum_{n \geq 0} \overline{S}_{3}(4n+2) q^{n} = 2 \frac{f_{6}^{4}}{f_{3}^{4}} + 6q \frac{f_{6}^{7} f_{9}^{5}}{f_{3}^{11} f_{18}} + 12q^{2} \frac{f_{6}^{6} f_{9}^{2} f_{18}^{2}}{f_{3}^{10}} + 24q^{3} \frac{f_{6}^{5} f_{18}^{5}}{f_{3}^{9} f_{9}}.
\]
Congruences \eqref{thm4-cong2} and \eqref{thm4-cong3} are immediate results from the above equation. By collecting the terms of the form $q^{3n}$ from both sides of the above equation and then applying Lemma \ref{binomial} to the resulting equation, with $p=2$ and $k=2$, we see that
\begin{equation}
\sum_{n \geq 0} \overline{S}_{3}(12n+2) q^{n} \equiv 2 \frac{f_{2}^{4}}{f_{1}^{4}} + 24q \frac{f_{2}^{5} f_{6}^{5}}{f_{1}^{9} f_{3}} \equiv 2f_{1}^{4} \ (\mathrm{mod} \ 8). \label{g2}
\end{equation}
Congruence \eqref{thm4-cong1} follows from \eqref{g2}.
\end{proof}

\begin{theorem} \label{thm6} For all integers  $\alpha,n\geq 0$, we have
\begin{align}
\overline{S}_{3} \big( 9^{\alpha}(36n+33) \big) &\equiv  0 \ (\mathrm{mod} \ 6),  \label{thm6-cong1}\\
\overline{S}_{3} \big( 9^{\alpha}(108n+45) \big) &\equiv  0 \ (\mathrm{mod} \ 6). \label{thm6-cong2}
\end{align}
\end{theorem}

\begin{proof} Substituting \eqref{lemma1.6} into \eqref{a2} and then collecting the even powers from both sides of the resulting equation, we arrive at
\[
\sum_{n \geq 0} \overline{S}_{3}(4n+1) q^{n} = 2 \frac{f_{4}^{2} f_{6}^{2}}{f_{1}^{2} f_{3}^{4} f_{12}^{2}}.
\]
In view of \eqref{binomiallemma} and the above equation, with $p=3$ and $k=1$, we find that
\begin{equation}
\sum_{n \geq 0} \overline{S}_{3}(4n+1) q^{n} \equiv 2 \frac{f_{1} f_{6}^{6}}{f_{4} f_{3}^{5} f_{12}} \ (\mathrm{mod} \ 6).\label{ra1}
\end{equation}
Now, by using \eqref{lemma4.2} in \eqref{ra1} and then collecting the terms of the form $q^{3n+2}$ from both sides of the resulting equation, we get
\begin{equation}
\sum_{n \geq 0} \overline{S}_{3}(12n+9) q^{n} \equiv 4 \frac{f_{2}^{8} f_{3} f_{12}^{3}}{f_{1}^{5} f_{4}^{5} f_{6}^{2}} \equiv 4 \frac{f_{1} f_{4} f_{6} f_{12}}{f_{2} f_{3}}= 4 \psi(-q) \frac{f_{6} f_{12}}{f_{3}} \ (\mathrm{mod} \ 6).\label{ra2}
\end{equation}
Employing \eqref{lemma8.2} in \eqref{ra2} and then extracting the terms of the form $q^{3n+j}$ for $j=0,2$, we get
\begin{align}
\sum_{n \geq 0} \overline{S}_{3}(36n+9) q^{n} &\equiv 4\frac{f_{4}^{2} f_{6}^{5}}{f_{2} f_{3}^{2} f_{12}^{2}} \ (\mathrm{mod} \ 6), \label{ra3}\\
\overline{S}_{3}(36n+33)  &\equiv 0  \ (\mathrm{mod} \ 6). \label{ra4}
\end{align}
Now, by using \eqref{lemma8.1} in \eqref{ra3}, with the transformation of $q^2$ to $q$, we find that
\[
\sum_{n \geq 0} \overline{S}_{3}(36n+9) q^{n} \equiv 4 \frac{f_{6}^{4} f_{18}^{2}}{f_{3}^{2} f_{12} f_{36}} + 4q^{2} \frac{f_{6}^{5} f_{36}^{2}}{f_{3}^{2} f_{12}^{2} f_{18}} \pmod 6.
\]
That is, if we extract the terms involving $q^{3n+j}$ for $j=1,2$ from the above equation, we obtain 
\begin{align}
\overline{S}_{3}(108n+45) &\equiv 0 \ (\mathrm{mod} \ 6), \label{ra5}\\
\sum_{n \geq 0} \overline{S}_{3}(108n+81) q^{n} &\equiv 4 \frac{f_{2}^{5} f_{12}^{2}}{f_{1}^{2} f_{4}^{2} f_{6}} \equiv 4  \frac{f_{1} f_{2} f_{6} f_{12}}{f_{2} f_{3}} = 4 \psi(-q) \frac{f_{6} f_{12}}{f_{3}} \ (\mathrm{mod} \ 6). \label{ra6}
\end{align}
Once again, by using \eqref{lemma8.2} in \eqref{ra6} and then collecting the terms of the form $q^{3n+j}$ for $j=0,2$, we get
\begin{align}
\sum_{n \geq 0} \overline{S}_{3}(324n+81) q^{n} &\equiv 4\frac{f_{4}^{2} f_{6}^{5}}{f_{2} f_{3}^{2} f_{12}^{2}} \ (\mathrm{mod} \ 6), \label{ra7}\\
\overline{S}_{3}(324n+297)  &\equiv 0  \ (\mathrm{mod} \ 6). \label{ra8}
\end{align}
Similarly, by substituting \eqref{lemma8.1} in \eqref{ra7}, with the transformation of $q^2$ to $q$ and then extracting the powers of the form $q^{3n+j}$ for $j=1,2$, we find that
\begin{align}
\overline{S}_{3}(972n+405) &\equiv 0 \ (\mathrm{mod} \ 6), \label{ra9}\\
\sum_{n \geq 0} \overline{S}_{3}(972n+729) q^{n} &\equiv 4 \frac{f_{2}^{5} f_{12}^{2}}{f_{1}^{2} f_{4}^{2} f_{6}} \equiv 4  \frac{f_{1} f_{2} f_{6} f_{12}}{f_{2} f_{3}} = 4 \psi(-q) \frac{f_{6} f_{12}}{f_{3}} \ (\mathrm{mod} \ 6). \label{ra10}
\end{align}
Now, from \eqref{ra2}, \eqref{ra6} and \eqref{ra10}, we deduce that for all positive integers $n$ and $\alpha$,
\begin{equation}
\sum_{n \geq 0} \overline{S}_{3} \big( 9^{\alpha} (12n+9) \big) q^{n} \equiv 4 \psi(-q) \frac{f_{6} f_{12}}{f_{3}}  \ (\mathrm{mod} \ 6), \label{ra11}
\end{equation}
and from \eqref{ra3} and \eqref{ra7}, we deduce that
\begin{equation}
\sum_{n \geq 0} \overline{S}_{3} \big( 9^{\alpha} (36n+9) \big) q^{n} \equiv 4 \frac{f_{4}^{2} f_{6}^{5}}{f_{2} f_{3}^{2} f_{12}^{2}} \ (\mathrm{mod} \ 6). \label{ra12}
\end{equation}
Therefore, congruence \eqref{thm6-cong1} follows from \eqref{ra4}, \eqref{ra8} and \eqref{ra12}. Meanwhile, congruence \eqref{thm6-cong2} follows from  \eqref{ra5}, \eqref{ra9} and \eqref{ra11}.
\end{proof}

\begin{theorem}\label{thm:mf}
    Let $j,n\geq 0$ be integers. For each $1\leq i\leq j+1$, let $p_1, p_2, \ldots, p_{j+1}$ be primes such that $p_1\geq 5$ and $p_i\not \equiv 0\pmod 6$. Then, for any integer $k\not \equiv 0\pmod{p_{j+1}}$ we have
    \[
    \sst_3(12p_1^2p_2^2\cdots p_{j+1}^2n+(12pk+2p_{j+1})p_1^2p_2^2\cdots p_j^2p_{j+1})\equiv 0 \pmod8.
    \]
\end{theorem}

\begin{proof}
    From \eqref{thm4-cong1}, we have
    \[
    \sum_{n\geq 0}\sst_3(12n+2)q^n\equiv 2f_1^4 \pmod 8,
    \]
    which implies
    \begin{equation}\label{eq:mf-1}
        \sum_{n\geq 0}\sst_3(12n+2)q^{6n+1}\equiv 2qf_{6}^4\equiv 2\eta^4(6z)\pmod 8.
    \end{equation}
    By using Theorem \ref{thm2.3} we see that $\eta^4(6z)\in S_2\left(\Gamma_0(36),\left(\frac{6^4}{\bullet}\right)\right)$, which means $\eta^4(6z)$ has a Fourier series expansion, say $\sum_{n\geq 1}a(n)q^n$. Clearly, we have
    \begin{equation}\label{eq:mf-2}
    a(n)=0 \quad \text{if} \quad n\not\equiv 1\pmod{6}, \quad \text{for all $n\geq 0$}. 
    \end{equation}
    Thus, \eqref{eq:mf-1} gives us
    \begin{equation}\label{eq:mf}
        \sst_3(12n+2)\equiv 2a(6n+1) \pmod 8.
    \end{equation}
    From \cite[pp. 4853]{Mar96} we know that $\eta^4(6z)$ is a Hecke eigenform, so we have
    \[
    \eta^4(6z)|T_p=\sum_{n\geq 1}\left(a(pn)+p\left(\frac{6^4}{p}\right)a\left(\frac{n}{p}\right)\right)
    q^n=\lambda(p)\sum_{n\geq 1}a(n)q^n,
    \]
    where the Legendre symbol $\left(\frac{6^4}{p}\right)=1$. Comparing the coefficients of $q^n$ in the above, we obtain
    \begin{equation}\label{eq:mf-3}
        a(pn)+pa\left(\frac{n}{p}\right)=\lambda(p)a(n).
    \end{equation}
    Clearly $a(p)=\lambda(p)$ since $a(1)=1$ and $a(1/p)=0$. Also, $\lambda(p)=0$ for all $p\equiv 1\pmod 6$ due to \eqref{eq:mf-2}. We get from \eqref{eq:mf-3}, for all $p\not \equiv 1 \pmod 6$,
    \begin{equation}\label{eq:mf-4}
        a(pn)+pa\left(\frac{n}{p}\right)=0.
    \end{equation}

If $p\nmid n$, then replacing $n$ by $pn+r$ with $\gcd(n,r)=1$ in \eqref{eq:mf-4}, we obtain
\begin{equation}
    a(p^2n+pr)=0.
\end{equation}
Substituting $n$ by $6n-pr+1$ in the above and then using \eqref{eq:mf} we obtain
\begin{equation}\label{eq:mf-19}
    \sst_3(12p^2n+2p^2+2pr(1-p^2))\equiv 0 \pmod 8.
\end{equation}

If $p|n$, then replacing $n$ by $pn$ in \eqref{eq:mf-4}, we obtain
\[
a(p^2n)=-pa(n).
\]
Again, using \eqref{eq:mf}, from the above, we obtain
\begin{equation}\label{eq:mf-21}
    \sst_3(12p^2n+2p^2)\equiv -p\sst_3(12n+2) \pmod8.
\end{equation}

    We see that $\gcd\left(\frac{1-p^2}{6},p\right)=1$, as $r$ runs over residue system excluding multiples of $p$, so does $\frac{(1-p^2)r}{6}$, thus for $p\nmid k$, we rewrite \eqref{eq:mf-19} as
    \begin{equation}\label{eq:mf-22}
        \sst_3(12p^2n+2p^2+12pk)\equiv 0 \pmod 8.
    \end{equation}
Let $p_i\geq 5$ be primes such that $p_1\not\equiv 1 \pmod 6$. We note that
\[
6p_1^2p_2^2\cdots p_j^2n+p_1^2p_2^2\cdots p_j^2=6p_1^2\left(p_2^2\cdots p_j^2n+\frac{p_2^2\cdots p_j^2-1}{6}\right)+p_1^2.
\]
Using \eqref{eq:mf-21} repeatedly, we obtain
\begin{equation}\label{eq:mf-23}
    \sst_3(12p_1^2p_2^2\cdots p_j^2n+2p_1^2p_2^2\cdots p_j^2)\equiv (-1)^jp_1p_2\cdots p_j\sst_3(12n+2)\pmod 8.
\end{equation}
Now, let $k\not\equiv 0 \pmod{p_{j+1}}$, from \eqref{eq:mf-22} and \eqref{eq:mf-23} we obtain
\[
\sst_3(12p_1^2p_2^2\cdots p_{j+1}^2n+(12k+2p_{j+1})p_1^2p_2^2\cdots p_j^2p_{j+1})\equiv 0 \pmod8.
\]
and this proves our result.
\end{proof}

\begin{corollary}
    Let $j, n \geq 0$ be integers. Let $p\geq 5$ be a prime such that $p\equiv 5\pmod 6$. Then, we have
    \[
    \sst_3(12p^{2j+2}n+12p^{2j+1}k+p^{2j+2})\equiv 0 \pmod 8,
    \]
    whenever $k\not\equiv 0\pmod p$.
\end{corollary}
\begin{proof}
    Putting $p_1=p_2=\cdots =p_{j+1}=p$ in Theorem \ref{thm:mf} gives us the corollary.
\end{proof}

\begin{theorem}\label{thm:mf-2}
    Let $j>0$ be an integer and $p$ be a prime such that $p\equiv 5\pmod 6$. Let $r\geq 0$ be an integer such that $p|12r+10$, then we have
    \[
    \sst_3(12p^{j+1}n+12pr+10p) \equiv -p\sst_3\left(12p^{j-1}n+\frac{12r+10}{p}\right) \pmod8.
    \]
\end{theorem}

\begin{proof}
    From \eqref{eq:mf-4}, replacing $n$ by $6n+5$ we obtain, for any prime $p\equiv 5 \pmod 6$
    \[
    a(6pn+5p)=-pa\left(\frac{6n+5}{p}\right).
    \]
    Now replacing $n$ by $p^jn+r$ with $p\nmid r$ in the above, we get
    \[
    a(6p^{j+1}n+6pr+5p)=-pa\left(6p^{j-1}n+\frac{6r+5}{p}\right).
    \]
    Using \eqref{eq:mf} and the above equation we obtain the desired result.
\end{proof}

\begin{corollary}
    Let $j>0$ be an integer and $p$ be a prime such that $p\equiv 5\pmod6$. Then, we have
    \[
    \sst_3(12p^{2j}n+p^{2j})\equiv (-p)^j\sst_3(12n+2)\pmod8.
    \]
\end{corollary}

\begin{proof}
    Let $p$ be as in the statement. Then we choose $r\geq 0$ an integer such that $12r+10=p^{2j-1}$. Substituting $j$ by $2j-1$ in Theorem \ref{thm:mf-2} we get
    \begin{align*}
        \sst_3(12p^{2j}n+2p^{2j})&\equiv -p\sst_3(12p^{2j-2}n+p^{2j-2}) \pmod 8.
    \end{align*}
    Repeatedly using this gives us the result.
\end{proof}

\section{Congruences for $\overline{S}_{9}(n)$} \label{sec3}
In this section, we derive several Ramanujan-like congruences and congruence families for the function $\overline{S}_{9}(n)$ modulo $4,6,8,9$ and $12$, along with a number of generating function dissections.

\begin{theorem} \label{thm7} We have the following dissections:
\begin{align}
&\sum_{n\geq 0} \overline{S}_{9}(6n) q^{n}=  1 + 12q \dfrac{f_{2} f_{6}^{5}}{f_{1}^{5} f_{3}}, \label{dis1}\\
&\sum_{n\geq 0} \overline{S}_{9}(6n+1) q^{n}=  2\dfrac{f_{2}^{6} f_{3}^{4}}{f_{1}^{8} f_{6}^{2}},  \label{dis2}\\
&\sum_{n\geq 0} \overline{S}_{9}(6n+2) q^{n}=  \dfrac{f_{2}^{3} f_{3}^{5}}{f_{1}^{7} f_{6}} + \dfrac{f_{1} f_{6}^{3}}{f_{2} f_{3}^{3}} + 16q \dfrac{f_{6}^{8}}{f_{1}^{4} f_{3}^{4}},  \label{dis3}\\
&\sum_{n\geq 0} \overline{S}_{9}(6n+3) q^{n}=  4\dfrac{f_{2}^{5} f_{3} f_{6}}{f_{1}^{7}}, \label{dis4}\\
&\sum_{n\geq 0} \overline{S}_{9}(6n+4) q^{n}=  6\dfrac{f_{2}^{2} f_{3}^{2} f_{6}^{2}}{f_{1}^{6}}, \label{dis5}\\
&\sum_{n\geq 0} \overline{S}_{9}(6n+5) q^{n}=  8\dfrac{f_{2}^{4} f_{6}^{4}}{f_{1}^{6} f_{3}^{2}}. \label{dis6}
\end{align}
\end{theorem}

\begin{remark}
We note that Ajeyakumar, Sumanth Bharadwaj and Chandankumar \cite[Theorem 3.2]{arxiv} also found \eqref{dis2}, \eqref{dis4} and \eqref{dis6}.
\end{remark}

\begin{proof} Setting $t=9$ in \eqref{maineq} and then replacing \eqref{lemma5.2} into the resulting equation, we arrive at
\begin{equation}
\sum_{n\geq 0} \overline{S}_{9}(n) q^{n} = \dfrac{f_{12}^{6}}{f_{2} f_{4} f_{6}^{2} f_{18} f_{36}} + 2q \dfrac{f_{4} f_{12}^{2} f_{36}}{f_{2}^{2} f_{18}^{2}}+ q^{2} \dfrac{f_{4}^{3} f_{6}^{2} f_{36}^{3}}{f_{2}^{3} f_{12}^{2} f_{18}^{3}}. \label{h1}
\end{equation}
Equating the even and the odd powers on both sides of equation \eqref{h1}, we get
\begin{align}
& \sum_{n\geq 0} \overline{S}_{9}(2n) q^{n} = \dfrac{f_{6}^{6}}{f_{1} f_{2} f_{3}^{2} f_{9} f_{18}} +q \dfrac{f_{2}^{3} f_{3}^{2} f_{18}^{3}}{f_{1}^{3} f_{6}^{2} f_{9}^{3}}, \label{h2}\\
& \sum_{n\geq 0} \overline{S}_{9}(2n+1) q^{n} = 2\dfrac{f_{2} f_{6}^{2} f_{18}}{f_{1}^{2} f_{9}^{2}}. \label{h3}
\end{align}
Now, by employing both \eqref{lemma2.2} and \eqref{lemma6.1} in \eqref{h2}, we find that
\[
\sum_{n\geq 0} \overline{S}_{9}(2n) q^{n} = \dfrac{f_{6}^{4} f_{9}^{8}}{f_{3}^{8} f_{18}^{4}} + q \bigg( \dfrac{f_{6}^{3} f_{9}^{5}}{f_{3}^{7} f_{9}} + \dfrac{f_{3} f_{18}^{3}}{f_{6} f_{9}^{3}} \bigg) + 6 q^{2} \dfrac{f_{6}^{2} f_{9}^{2} f_{18}^{2}}{f_{3}^{6}} + 4q^{3} \dfrac{f_{6} f_{18}^{5}}{f_{3}^{5} f_{9}} + 16q^{4} \dfrac{f_{18}^{8}}{f_{3}^{4} f_{9}^{4}}.
\]

The dissections \eqref{dis3} and \eqref{dis5} follow from the above equation by extracting the powers of the form $q^{3n+j}$ for $j=1,2$ from both sides. By extracting the terms of the form $q^{3n}$ from both sides of the above equation, we get
\begin{equation}
\sum_{n\geq 0} \overline{S}_{9}(6n) q^{n}=  \dfrac{f_{2}^{4} f_{3}^{8}}{f_{1}^{8} f_{6}^{4}} + 4q \dfrac{f_{2} f_{6}^{5}}{f_{1}^{5} f_{3}}. \label{h4}
\end{equation}
By substituting \eqref{lemma9.1} into \eqref{h4}, we obtain the dissection \eqref{dis1}. Meanwhile, by employing \eqref{lemma4.1} in \eqref{h3}, we obtain
\[
\sum_{n\geq 0} \overline{S}_{9}(2n+1) q^{n} = 2\dfrac{f_{6}^{6} f_{9}^{4}}{f_{3}^{8} f_{18}^{2}} + 4 q \dfrac{f_{6}^{5} f_{9} f_{18}}{f_{3}^{7}} +8q^{2} \dfrac{f_{6}^{4} f_{18}^{4}}{f_{3}^{6} f_{9}^{2}}.
\]
Similarly, dissections \eqref{dis2}, \eqref{dis4}, and \eqref{dis6} follow from the above equation by extracting the powers of the form $q^{3n+j}$ for $j=0,1,2$ from both sides.
\end{proof}

\begin{corollary}\label{ccr} For all $n\geq 0$, we have
\begin{align*}
 &\overline{S}_{9}(6n+3) \equiv 0 \ (\mathrm{mod} \ 4),\\
 &\overline{S}_{9}(6n+4) \equiv 0 \ (\mathrm{mod} \ 6),\\
 &\overline{S}_{9}(6n+5) \equiv 0 \ (\mathrm{mod} \ 8),\\
 &\overline{S}_{9}(6n+6) \equiv 0 \ (\mathrm{mod} \ 12).
\end{align*}
\end{corollary}

\begin{theorem}
For all $n\geq 0$, we have
\begin{align}
\sst_9(12n+11)&\equiv 0 \pmod{16},\label{pmod16}\\
\sst_9(24n+23)&\equiv 0 \pmod{32}.\label{pmod32}
\end{align}
\end{theorem}

\begin{remark}
We note that Ajeyakumar, Sumanth Bharadwaj and Chandankumar \cite[Eq. (3.12)]{arxiv} also proved \eqref{pmod32}.
\end{remark}

\begin{proof}
Using \eqref{dis6}, \eqref{diss-2-f14} and \eqref{lemma1.6}, we have
\[
\sum_{n\geq 0}\sst_9(6n+5)q^n=8\dfrac{f_{2}^{4} f_{6}^{4}}{f_{1}^{6} f_{3}^{2}}=8f_2^4f_6^4\left(\frac{f_4^{14}}{f_2^{14}f8^4}+4q\frac{f_4^2f_8^4}{f_2^{10}}\right)\left(\frac{f_{8}^{2} f_{12}^{5}}{f_{2}^{2} f_{4} f_{6}^{4} f_{24}^{2}}+q \frac{f_{4}^{5} f_{24}^{2}}{f_{2}^{4} f_{6}^{2} f_{8}^{2} f_{12}}\right)^2.
\]
Reducing the above modulo $32$, we obtain
\[
\sum_{n\geq 0}\sst_9(6n+5)q^n\equiv 8f_2^4f_6^4\left(\frac{f_4^{14}}{f_2^{14}f8^4}\right)\left(\frac{f_{8}^{2} f_{12}^{5}}{f_{2}^{2} f_{4} f_{6}^{4} f_{24}^{2}}+q \frac{f_{4}^{5} f_{24}^{2}}{f_{2}^{4} f_{6}^{2} f_{8}^{2} f_{12}}\right)^2 \pmod{32}.
\]
Now, extracting the terms involving the odd powers of $q$, we have
\[
\sum_{n\geq 0}\sst_9(12n+11)q^n\equiv 16\frac{f_2^{18}f_6^4}{f_1^{16}f_3^2f_4^4} \pmod{32}.
\]
The above proves \eqref{pmod16}.

Again, notice the above is equivalent to
\begin{align*}
\sum_{n\geq 0}\sst_9(12n+11)q^n &\equiv 16\frac{f_2^{18}f_6^4}{f_2^{8}f_6f_4^4} \pmod{32}\\
&\equiv 16\frac{f_2^{10}f_6^3}{f_4^4} \pmod{32}.
\end{align*}
Extracting the terms involving the odd powers of $q$ from the above yields \eqref{pmod32}.
\end{proof}

\begin{theorem} \label{thm8}  For all nonnegative integers $\alpha$ and  $n$, we have
\begin{equation*}
\overline{S}_{9} \bigl(6 \cdot 5^{2\alpha+2}n + 6 \cdot 5^{2\alpha+1}i + 5^{2\alpha+2}  \bigl) \equiv 0 \ (\mathrm{mod} \ 6), 
\end{equation*}
for all $1\leq i \leq 4$.
\end{theorem}

\begin{proof} In view of \eqref{dis2} and \eqref{binomiallemma}, with $p=3$ and $k=1$, we find that
\[
\sum_{n \geq 0} \overline{S}_{9}(6n+1) q^{n}= 2\dfrac{f_{2}^{6} f_{3}^{4}}{f_{1}^{8} f_{6}^{2}} \equiv 2f_{1}^{4} \ (\mathrm{mod} \ 6).
\]
Now, by using the same steps in Theorem \ref{thm3}, we obtain the desired result.
\end{proof}

\begin{theorem} \label{thm9} For any prime $p\equiv 5 \ (\mathrm{mod} \ 6)$, $\alpha \geq 0$, and $n\geq 0$, we have
\[
\overline{S}_{9} \left(6p^{2\alpha+1}(pn+i)+3p^{2\alpha+2}\right)\equiv 0 \ (\mathrm{mod} \ 12),
\]
for all $1\leq i \leq p-1$.
\end{theorem}

\begin{proof} In view of \eqref{dis4} and \eqref{binomiallemma}, with $p=3$ and $k=1$, we find that
\[
\sum_{n \geq 0} \overline{S}_{9}(n) q^{n}= 4\dfrac{f_{2}^{5} f_{3} f_{6}}{f_{1}^{7}} \equiv 4\frac{f_{2}^{2} f_{6}^{2}}{f_{1} f_{3}}  \ (\mathrm{mod} \ 12), 
\]
and by using the definition of $\psi(q)$ we see that
\[
\sum_{n \geq 0} \overline{S}_{9}(6n+3) q^{n} \equiv 4\psi(q) \psi(q^{3}) \ (\mathrm{mod} \ 12).
\]
Now, by using the same steps in Theorem \ref{thm2}, we obtain the desired result.
\end{proof}

\begin{theorem} For all $n\geq 0$, we have

\begin{align}
&\overline{S}_{9}(12n+4)  \equiv 0 \ (\mathrm{mod} \ 6), \label{l1}\\
&\overline{S}_{9}(12n+10) \equiv 0 \ (\mathrm{mod} \ 18). \label{l2}
\end{align}
\end{theorem}

\begin{proof} Substituting \eqref{lemma1.4} into \eqref{dis5}, we get
\begin{equation}
\sum_{n\geq 0} \overline{S}_{9}(6n+4) q^{n}  = 6\dfrac{f_{4}^{12} f_{6}^{8}}{f_{2}^{16} f_{12}^{4}} + 18q \dfrac{f_{4}^{8} f_{6}^{6}}{f_{2}^{14}} +54q^{2} \dfrac{f_{4}^{4} f_{6}^{4} f_{12}^{4}}{f_{2}^{12}}. \label{l3}
\end{equation}
Congruences \eqref{l1} and \eqref{l2} are true from equation \eqref{l3}.
\end{proof}

\begin{theorem} For all nonnegative integers $\alpha$ and  $n$, we have
\[
\overline{S}_{9}(3 \cdot 4^{\alpha+2} n + 10 \cdot 4^{\alpha+1} )  \equiv 0 \ (\mathrm{mod} \ 9).
\]
\end{theorem}

\begin{proof} By extracting the terms of the form $q^{2n}$ from both sides of equation \eqref{l3}, we get
\begin{equation}
\sum_{n\geq 0} \overline{S}_{9}(12n+4) q^{n}  = 6\dfrac{f_{2}^{12} f_{3}^{8}}{f_{1}^{16} f_{6}^{4}} +54q \dfrac{f_{2}^{4} f_{3}^{4} f_{6}^{4}}{f_{1}^{12}}. \label{l4}
\end{equation}
In view of \eqref{binomiallemma} and \eqref{l4}, with $p=3$ and $k=2$, we obtain
\begin{equation}
\sum_{n\geq 0} \overline{S}_{9}(12n+4) q^{n}  \equiv 6\dfrac{f_{2}^{12} f_{3}^{8}}{f_{1}^{16} f_{6}^{4}} \equiv 6 \frac{f_{1}^{2} f_{2}^{3} f_{3}^{2}}{f_{6}} \ (\mathrm{mod} \ 9). \label{l5}
\end{equation}
By using \eqref{lemma1.7} into \eqref{l5} and then extracting the odd terms from both sides of the resulting equation, we obtain
\begin{equation}
\sum_{n\geq 0} \overline{S}_{9}(24n+16) q^{n}  \equiv 6 \frac{f_{1}^{3} f_{2}^{2} f_{6}^{2}}{f_{3}} \ (\mathrm{mod} \ 9). \label{l6}
\end{equation}
By employing \eqref{lemma1.41} in \eqref{l6}, we get
\begin{equation}
\sum_{n\geq 0} \overline{S}_{9}(24n+16) q^{n}  \equiv 6 \frac{f_{2}^{2} f_{6}^{2} f_{4}^{3}}{f_{12}} -18 \frac{f_{2}^{4} f_{12}^{3}}{f_{4}} \ (\mathrm{mod} \ 9). \label{l7}
\end{equation}
If we extract the even and the odd terms from both sides of \eqref{l7} we get
\begin{align}
\sum_{n\geq 0} &\overline{S}_{9}(48n+16) q^{n}  \equiv 6\frac{f_{1}^{2} f_{2}^{3} f_{3}^{2}}{f_{6}} \ (\mathrm{mod} \ 9), \label{l8}\\
&\overline{S}_{9}(48n+40)  \equiv 0 \ (\mathrm{mod} \ 9). \label{l9}
\end{align}
Again, by using \eqref{lemma1.7} into \eqref{l8} and then extracting the odd terms from both sides of the resulting equation, we obtain
\begin{equation}
\sum_{n\geq 0} \overline{S}_{9}(96n+64) q^{n}  \equiv 6 \frac{f_{1}^{3} f_{2}^{2} f_{6}^{2}}{f_{3}} \ (\mathrm{mod} \ 9), \label{l10}
\end{equation}
and again, by using \eqref{lemma1.41} in \eqref{l10} and then extracting the even and the odd terms from both sides of the resulting equation, we see that
\begin{align}
\sum_{n\geq 0} &\overline{S}_{9}(192n+64) q^{n}  \equiv 6\frac{f_{1}^{2} f_{2}^{3} f_{3}^{2}}{f_{6}} \ (\mathrm{mod} \ 9), \label{l11}\\
&\overline{S}_{9}(192n+160)  \equiv 0 \ (\mathrm{mod} \ 9). \label{l12}
\end{align}
Our result follows from \eqref{l8}, \eqref{l9}, \eqref{l11} and \eqref{l12}.
\end{proof}

\subsection*{Funding Statement}

The first and the last author did not receive support from any organization for the submitted work. The second author is partially supported by a Start-Up Grant from Ahmedabad University, India (Reference No. URBSASI24A5).

\subsection*{Data Availability Statement} This manuscript contains no associated data.

\subsection*{Conflict of Interest} The authors declare that there are no known conflicts of interest.

\end{document}